\documentclass[12pt]{article}
\usepackage[utf8]{inputenc}
\usepackage{natbib}
\usepackage{eurosym}
\usepackage{amstext,amsthm}
\usepackage{amsmath}
\usepackage{amssymb,mathtools}
\usepackage[nointegrals]{wasysym} 
\usepackage{mathrsfs}
\usepackage{graphicx}
\usepackage{bbm}
\usepackage{tikz}
\usetikzlibrary{arrows, decorations.markings}
\usepackage{graphicx} 				
\usepackage{subfig}					
\usepackage{rotating}

\newtheorem{proposition}{Proposition}[section]

\newtheorem{theorem}{Theorem}
\newtheorem{lemma}[proposition]{Lemma}
\theoremstyle{definition}
\newtheorem{definition}[proposition]{Definition}
\newtheorem{remark}[proposition]{Remark}

\usepackage{a4}
\usepackage[normalem]{ulem}
\usepackage{color}

\DeclareMathAlphabet{\mathpzc}{OT1}{pzc}{m}{it}

\usepackage{times}

\newcommand\unnumberedfootnote[1]{ %
        \let\temp=\thefootnote %
        \renewcommand{\thefootnote}{}%
        \footnote{#1}%
        \let\thefootnote=\temp%
        \addtocounter{footnote}{-1}}

\begin{document}
\title{\LARGE The range of once-reinforced random walk in one
  dimension}

\thispagestyle{empty}

\author{{\sc by  P. Pfaffelhuber and J. Stiefel} \\[2ex]
  \emph{Albert-Ludwigs University Freiburg} } \date{\today}

\maketitle

\unnumberedfootnote{\emph{AMS 2000 subject classification.} {\tt 60J15}
  (Primary) {\tt, 60K99} (Secondary).}

\unnumberedfootnote{\emph{Keywords and phrases.} Reinforced random
  walk; range process; scaling limit}

\vspace*{-5ex}

\begin{abstract}
  \noindent
  We study once-reinforced random walk (ORRW) on $\mathbb Z$. For this
  model, we derive limit results on all moments of its range using
  Tauberian theory. 
\end{abstract}

\section{Introduction}
Interacting random walks are a class of (mostly non-Markovian) models
where the next step of (mostly simple) random walk depends on its
path. Some models tend to visit new sites with high probability. Such
models arose in chemical physics as a model for long polymer chains,
and are discussed in some detail in Chapter 6 of \cite{Lawler2013}. A
particularly prominent example is the self-avoiding walk, which visits
every site not more than once. Other models are the myopic (or true)
self-avoiding walk, which has higher chances to move to sites it
visited less than others. Another class of random walk models --
usually referred to as reinforced random walks -- visit sites (or walk
along edges) more likely they have already seen; see e.g.\
\cite{Davis1990} and \cite{Pemantle2007}.

Here, we study a model which appeared as the hungry random walk in the
Physics literature \citep{schilling2017clearing} for mimicking
chemotaxis. For the model in $d=1$, every site in $\mathbb Z$ contains
food, which is eaten up once the walker visits the site. In addition,
the walker rather visits sites with food, if it sees a neighboring
site containing food (with probability proportional to $e^\gamma$ for
some $\gamma \in \mathbb R$) than going the other direction (with
probability proportional to 1). Two models from the probability
literature are equivalent to the hungry random walk on $\mathbb Z$:
First, the true self-avoiding walk with generalized bond-repulsion, as
studied in \cite{Toth1993}, jumps along an edge $b$ with probabilities
proportional to
$\exp(-\gamma \cdot (\text{number of previous jumps along
  $b$})^\kappa)$, which equals the hungry random walk in the limit
$\kappa\to 0$. Second, the case $\gamma<0$ is equivalent to
once-reinforced random walk or ORRW: here, every edge has initial
weight~1, and once the walker goes along an edge, the weight changes
to $c>1$. The walker then chooses its next step according to the edge
weights. Apparently, this is the same as the hungry random walker for
$\gamma = -\log c<0$ on $\mathbb Z$. Recent literature on the ORRW has
focused on recurrence and transience on various graphs (see e.g.\
\citealp{Durrett2002, Dai2005, Sellke2006, KiousSido2018}). Here, we
rather stick to $\mathbb Z$ but aim at concrete formulas for the
asymptotics of the range of ORRW in Theorem~\ref{T1}. Our
analysis is based on a simple decomposition of the inverse of the
range process as given in \eqref{eq:SRWtimes}. Notably, we
cannot compute moments of ORRW itself. At least, we give some
heuristics of the variance in Remark~\ref{rem:variance}.

In studying the ORRW, we will not restrict ourselves to $c>1$, but to
$c>0$. A scaling limit of the ORRW in this case was studied in
\cite{Davis1996}, \cite{PermanWerner1997} and
\cite{CarmonaPetitYor1998}. More precisely, it was shown (see Theorem
1.2 in \citealp{Davis1996}) that (choose $\alpha = -\beta = 1-c$) for
$1/2<c<3/2$, the sequence $X^n = (X_{nt}/\sqrt{n})_{t\geq 0}$ has a
limit $Y$ as $n\to\infty$ which solves
\begin{align}
  \label{eq:scaling}
  Y_t = B_t + (1-c) (\sup\{Y_s: s\leq t\} + |\inf\{Y_s: s\leq t\}|).  
\end{align}
More connections of our results to this equation are discussed in
Remark~\ref{rem:perBB}.

~

The paper is organized as follows: In the next Section, we give our
main result, Theorem~\ref{T1}, which gives asymptotics of all moments
of the range of the ORRW. Section~\ref{ss:31} contains some
preliminary steps for our proofs. The proof of Theorem~\ref{T1} is
given in Section~\ref{ss:32}.

\section{Results}
\label{S2}
\begin{definition}
  Let $c>0$ and $X=(X_n)_{n=0,1,2,...}$ be the stochastic process with
  $X_0=0$, and, for $n=0,1,2,...$, given $X_0,...,X_n$, and setting
  $M_n := \max_{k\leq n} X_k$ as well as $m_n:= \min_{k\leq n} X_k$,
  \begin{align*}
    & X_{n+1} = X_n + 1 \text{ with probability } = \begin{cases} \displaystyle\frac{1}{1+c} & \text{ if } X_n = M_n
      \\[2ex] \quad\displaystyle\frac{1}{2} & \text{ if }X_n \neq m_n, M_n,
      \\[2ex] \displaystyle\frac{c}{1+c} &\text{ if }X_n = m_n
    \end{cases}\\
    & \text{and }\mathbb P(X_{n+1}=X_n-1) = 1 - \mathbb P(X_{n+1} = X_n+1).
 \end{align*}
  In other words, $X_{n+1} = X_n \pm 1$ with probability proportional
  to $c$ or $1$, if $X$ has or has not visited $X_n\pm 1$ before time
  $n$. We call $X = (X_n)_{n=0,1,2,...}$ the {\em once-reinforced
    random walk} (or {\it ORRW}) on $\mathbb Z$ with parameter
  $c$. Its range by time $n$ is given by
  \begin{align}
    \label{eq:Rn}
    R_n := M_n - m_n = M_n + |m_n|. 
  \end{align} \qed
\end{definition}

\noindent
Note that only the case $c>1$ gives a reinforced walk (in the sense
that it visits previously seen sites more likely), while the walk has
self-avoiding properties for $0<c<1$. For $c=1$, it is just the
symmetric Bernoulli walk. Since our proofs work in all cases, we do
not distinguish them in the sequel.

~

\noindent
The range process $R = (R_n)_{n=0,1,...}$ is a nondecreasing process
with jumps of size 1, and is our main object of study. The following
ideas are essential to understand our approach. The random time
$S_k := \inf\{n: R_n = k\}$ is the first time the ORRW has range $k$
(such that $k\mapsto S_k$ is the generalized inverse of
$n\mapsto R_n$), and $T_i := S_{i+1} -S_{i}$ is the time between
$R_n=i$ for the first time and $R_n=i+1$. In order to study $T_i$ (for
$i=1,2,...$), we note that $T_i = 1$ with probability
$1/(1+c)$. Otherwise, the random walk moves within its range (which is
$i$ at that time) until it first hits its maximum or minumum, which
takes time $\tau_i$, the hitting time of $\{-1, i-1\}$ of a simple
random walk starting in~0. Again, the chance to increase the range is
$1/(1+c)$ etc. The number of times the random walk needs a chance of
$1/(1+c)$ to increase its range is a geometrically distributed random
variable $Y_i$ with parameter $1/(1+c)$. (Note that $Y_i=0$ is
possible, i.e.\ we must use the shifted geometrical distribution.) In
total, this gives
\begin{equation}\label{eq:SRWtimes}
  S_k = 1 + \sum_{i=1}^{k-1} T_i, \quad k=1,2,... \text{ with }
  T_i = 1 +  \sum_{j=1}^{Y_i} \left(1+\tau_i^j\right),\quad i = 1,2,...
\end{equation}  
(where we define the empty sum to be~0). Here, $\tau_i^k, k= 1,2,...$
are independent and identically distributed as $\tau_i$ above and also
independend from $Y_i, i=1,2,...$ Using \eqref{eq:SRWtimes}, we can
compute the generating function of $S_k$, (see Lemma~\ref{l:gentaui})
and then use $\mathbb P(R_n>k) = \mathbb P(S_k<n)$ in order to obtain
results on $R_n$ (see Lemma~\ref{l:Hls} and Proposition~\ref{P:Hell}).

~

\noindent
We are now ready to formulate our main result, which will be proved in
Section~\ref{ss:32}.  Throughout, we will write $a_n \sim b_n$ if
$a_n/b_n \xrightarrow{n\to\infty} 1.$

\begin{theorem}[Asymptotic moments of the range]\label{T1}
  Let $R_n$ be as in \eqref{eq:Rn} the range of the ORRW with
  parameter $c>0$. Then,
  \begin{align}\label{eq:T1}
    \mathbb E\Big[\Big(\frac{R_n}{\sqrt n}\Big)^{\ell}\Big]
    & \sim \frac{1}{2^{(\ell-2)/2}
      \Gamma(\ell/2)} \cdot J_\ell(c), \qquad \ell=1,2,...
      \intertext{with}
      J_\ell(c) & := 2^{2c} 
                  \int_0^\infty x^{\ell-1} \Big(\frac{e^{x}}{(e^x+1)^{2}}\Big)^cdx.\notag
  \end{align}
  In particular,
  \begin{align*}
    \mathbb E\Big[\frac{R_n}{\sqrt{n}}\Big]
    & \sim
      \sqrt{\frac{2}{\pi}} \cdot J_1(c)  \qquad \text{ and }\qquad 
      \mathbb E\Big[\frac{R_n^2}{n}\Big] \sim
      J_2(c).        
  \end{align*}
\end{theorem}

\begin{remark}[The range for $c=1$]\label{rem:c1}
  The ORRW with $c=1$ equals the the symmetric Bernoulli walk. In this
  case, several results have been obtained for the range. An early
  example is \cite{Feller1951}, who states in his (1.4) that
  \begin{align*}
    \mathbb E\Big[\frac{R_n}{\sqrt{n}}\Big] \sim \sqrt{\frac 8 \pi}, \qquad  \qquad 
    \mathbb E\Big[\Big(\frac{R_n}{\sqrt{n}}\Big)^2\Big] \sim 4\log 2.
  \end{align*}
  In this case, we compute
  \begin{align}\label{eq:Fel0}
    J_1(1) & = 4 \int_0^\infty \frac{e^x}{(e^x + 1)^2} dx = -\frac{4}{e^x+1}\Big|_{0}^\infty = 2
  \end{align}
  and Feller's result for the expectation follows from
  \eqref{eq:T1}. Moreover, using integration by parts,
  \begin{align}\label{eq:Fel1}
    J_2(1) & = 4 \int_0^\infty x \frac{e^x}{(e^x + 1)^2} dx = - 4 \frac{x}{e^x+1}\Big|_0^\infty
             + 4 \int_0^\infty \frac{1}{e^x+1} dx
    \\ & \stackrel{y=e^x} = 4 \int_1^\infty \frac{1}{y} - \frac{1}{y+1} dy = 4\int_1^2 \frac 1y dy
         = 4\log 2,\notag
  \end{align}
  which gives Feller's result for the second moment.

  In addition to these limiting results, \cite{Vallois1993} and
  \cite{Vallois1995} have computed the generating function as well as
  expectation and variance for $S_k$, given through
  \begin{align*}
    \mathbb E[S_k] = \binom{k+1}{2}, \qquad \qquad \mathbb V[S_k] = \frac{(k-1) k (k+1)(k+2)}{12}.
  \end{align*}
  These results can as well be obtained as follows: Modyfing
  \eqref{eq:SRWtimes} for the case $c=1$, we can write
  \begin{align*}
    S_k = \sum_{i=0}^{k-1} \tau_{i+2},
  \end{align*}
  where $\tau_{i+2}$ is the hitting time of $\{-1, i+1\}$ of a random
  walk starting in~0. This holds since the range increases if and only
  if such a hitting time is observed. We note that $\tau_{i+2}$ is the
  duration of play of a symmetric Gambler's ruin starting with~1 and a
  total of $i+2$ units. It is a classical result that
  $\mathbb E[\tau_{i+2}] = i+1$, and
  $\mathbb V[\tau_{i+2}] = \frac{(i+2)(i+1)i}{3}$ was e.g.\ derived in
  \cite{Bach1997}.  Summing then gives Vallois' results.
\end{remark}

\begin{remark}[$J_\ell(c)$ for integer-valued $c$]
  If $c$ is an integer, i.e.\ $c=1,2,...$, the calculations from
  \eqref{eq:Fel0} and \eqref{eq:Fel1} can be generalized and lead to
  specific expressions for $J_\ell(c)$. We just give the necessary
  steps for $\ell=1,2$, which can then be generalized for larger
  $\ell$. First, note that a straight-forward calculation shows that,
  for $c=1,2,...$
  $$ \frac{d}{dx}  \sum_{j=0}^{c-1} \binom{c-1}{j} (-1)^{c-1-j}\frac{1}{j-2c+1} (e^x+1)^{j-2c+1}
  = \Big(\frac{e^x}{(e^x+1)^2}\Big)^c.$$ Hence,
  \begin{align*}
    J_1(c) & = \sum_{j=0}^{c-1} \binom{c-1}{j} \frac{1}{j-2c+1} 2^{j+1}(-1)^{c-j}.
  \end{align*}
  Next,
  \begin{align*}
    \frac{d}{dx} \Big(x & - \log(e^x+1) + \sum_{i=1}^{j-1} \frac{1}{i(e^x+1)^i}\Big) = 1 - \frac{e^x}{e^x+1} \sum_{i=0}^{j-1} \frac{1}{(e^x+1)^{i}} = \frac{1}{(e^x+1)^j},
  \end{align*}
  and therefore, using integration by parts,
  \begin{align*}
    J_2(c) & = 2^{2c} \int_0^\infty x\Big(\frac{e^x}{(e^x+1)^2}\Big)^c dx
    \\ & = 2^{2c} \sum_{j=0}^{c-1}\binom{c-1}{j}(-1)^{c-j} \frac{1}{j-2c+1}
         \int_0^\infty \frac{1}{(e^x+1)^{2c-j-1}}dx
    \\ & = 2^{2c} \sum_{j=0}^{c-1}\binom{c-1}{j}
         (-1)^{c-j} \frac{1}{j-2c+1} \Big(x - \log(e^x+1) + \sum_{i=1}^{2c-j-2} \frac{1}{i(e^x+1)^i}\Big)\Big|_{x=0}^\infty
    \\ & = 2^{2c} \sum_{j=0}^{c-1}\binom{c-1}{j}
         (-1)^{c-j-1} \frac{1}{2c-j-1} \Big(\log(2) - \sum_{i=1}^{2c-j-2} \frac{1}{i2^i}\Big).
  \end{align*}
  For higher $\ell$, more steps using integration by parts are
  necessary.  \qed
\end{remark}

\begin{remark}[Scaling limit of ORRW]\label{rem:perBB}
  Theorem~\ref{T1} can be understood as results for the scaling limit
  \begin{align*}
    X_{n . }/\sqrt{n} \xRightarrow{n\to\infty} Y = (Y_t)_{t\geq 0},
  \end{align*}
  where $Y$ is given in~\eqref{eq:scaling}; see \cite{Davis1996} for
  the corresponding limit result. By this, we mean that the range
  $\hat R$ of $Y$ satisfies
  \begin{align*}
    \mathbb E[\hat R_t^\ell]
    & = \lim_{n\to\infty} \mathbb E[(R_{tn}/\sqrt{n})^\ell] = \frac{1}{2^{(\ell-2)/2}
      \Gamma(\ell/2)} \cdot J_\ell(c) \cdot t^{\ell/2}.
  \end{align*}
  While the convergence above was only shown for $1/2 < c < 3/2$ in
  \cite{Davis1996}, we briefly argue how this converges comes about:
  Note that $Y$ solves \eqref{eq:scaling} iff
  \begin{align}
    \label{eq:BB}
    Y_t - (1-c)(\sup\{Y_s: s\leq t\} + |\inf\{Y_s: s\leq t\}|) \text{ is a
    Brownian motion.}
  \end{align}  
  For the ORRW, note that 
  \begin{align*}
    N_n & := X_n - \frac{1-c}{1+c} \Big(\sum_{k<n} 1_{X_k = M_k} - \sum_{k<n} 1_{X_k = m_k}\Big)
          \intertext{is a martingale, since (wp$=$with probability)}
          N_{n+1}^1 - N_{n}^1 & = \begin{cases} \pm 1 & \text{ wp }\tfrac 12 \text{ if }X_n \neq m_n, M_n,
            \\ 1 - \frac{1-c}{1+c} = \frac{2}{1+c} & \text{ wp } \frac{c}{1+c} \text{ if } X_n=M_n,
            \\ -1 - \frac{1-c}{1+c} = \frac{-2}{1+c}& \text{ wp } \frac{1}{1+c} \text{ if } X_n=M_n,
            \\ +1 + \frac{1-c}{1+c} = \frac{2}{1+c}& \text{ wp } \frac{c}{1+c} \text{ if } X_n=m_n,
            \\ - 1 + \frac{1-c}{1+c} = \frac{-2c}{1+c}& \text{ wp } \frac{1}{1+c} \text{ if } X_n=m_n.
          \end{cases}
  \end{align*}
  For large $n$, we have that
  $\frac{1}{1+c} \sum_{k<nt} 1_{X_k=M_k} \sim M_{nt}$ by the law of
  large numbers, since every time with $X_k = M_k$ there is an
  independent chance of $1/(1+c)$ of increasing $M$. Moreover, a
  straight-forward calculation gives that $N$ has quadratic variation
  $$ \langle N\rangle_{nt} = nt - \Big(\frac{1-c}{1+c}\Big)^2
  \Big(\sum_{k<nt} 1_{X_k = M_k} + \sum_{k<nt} 1_{X_k = m_k}\Big) \sim
  nt - \frac{(1-c)^2}{1+c} R_{nt},$$ where the $\sim$ follows from
  $\tfrac{1}{1+c}(\sum_{k<nt} 1_{X_k = M_k} + \sum_{k<nt} 1_{X_k =
    m_k}) \sim R_n$ by the same argument using the law of large
  numbers as above. Since $\mathbb E[R_{nt}] = O(\sqrt{nt})$,
  as we have shown in Theorem~\ref{T1}, we thus have that the limit of
  $\frac{N_{n.}}{\sqrt{n}}$ as $n\to\infty$ is the same as the limit
  of
  $$ \Big(\frac{X_{nt}}{\sqrt{n}} - (1-c)\Big(\frac{M_{nt}}{\sqrt{n}} + \frac{|m_{nt}|}{\sqrt{n}}\Big)\Big)_{t\geq 0},$$
  which must be a continuous martingale with quadratic variation~$t$
  by time $t$, i.e.\ a Brownian motion. This is enough to conclude
  that scaling limits of $X$ satisfy~\eqref{eq:BB}.
\end{remark}

\begin{remark}[\text{Towards $\mathbb V[X_n/\sqrt{n}]$}]
  \label{rem:variance}
  Although we are able to asymptotically compute all moments of
  $R_n/\sqrt n$ as in Theorem~\ref{T1}, we are unable to compute
  asymptotics of (even) moments of $X_n/\sqrt n$. At least, we now
  give some thoughts and bounds of asymptotics of
  $\mathbb V[X_n/\sqrt n] = \mathbb E[(X_n/\sqrt n)^2]$.
  We observe that
  \begin{align*}
    N_n & := X_n^2 - n - 2\frac{1-c}{1+c}\sum_{k<n} |X_k|1_{X_k\in\{m_k,M_k\}},
  \end{align*}
  is a mean-zero-martingale, since (note that $m_k\leq 0 \leq M_k$)
  \begin{align*}
    N_{n+1}^1 - N_{n}^1 & = \begin{cases} \pm 2X_n &
      \text{ wp }\tfrac 12 \text{ if }X_n \neq m_n, M_n,
      \\ 2X_n - 2\frac{1-c}{1+c}X_n = 4X_n \frac{c}{1+c} &
      \text{ wp } \frac{1}{c+1} \text{ if } X_n=M_n,
      \\ -2X_n - 2\frac{1 - c}{1+c}X_n = -4X_n\frac{1}{1+c} &
      \text{ wp } \frac{c}{1+c} \text{ if } X_n=M_n,
      \\ 4X_n \frac{1}{1+c} & \text{ wp } \frac{c}{1+c} \text{ if } X_n=m_n,
      \\ -4X_n \frac{c}{1+c} & \text{ wp } \frac{1}{1+c} \text{ if } X_n=m_n.
    \end{cases}
  \end{align*}
  Now, for large $n$, we have
  $M_n = (1+c) \sum_{k<n} 1_{X_k = M_k} + o(\sqrt n)$ by the law of
  large numbers, hence
  \begin{align*}
    2\sum_{k<n} \mathbb E[M_k1_{X_k = M_k\}}] & = 2(1+c) \cdot \mathbb E\Big[\sum_{\ell < k < n}
                                                1_{X_\ell = M_\ell} 1_{X_k = M_k}\Big] + o(\sqrt n)
    \\ & = (1+c) \cdot \mathbb E\Big[\Big(\sum_{k<n}1_{X_k = M_k}\Big)^2\Big] + O(\sqrt n) = \frac{1}{1+c} M_n^2 + O(\sqrt n),
  \end{align*}
  and by symmetry
  \begin{align*}
    \mathbb E[X_n^2] & = n + (1-c)\mathbb E[M_n^2 + m_n^2] + O(\sqrt{n}),
  \end{align*}
  which gives the intuitive result that $\mathbb E[X_n^2]\leq n$ for
  $c>1$ and $\mathbb E[X_n^2]\geq n$ for $c < 1$. Moreover, since
  $0 \leq (M_n - |m_n|)^2$ implies $2M_n |m_n| \leq M_n^2 + |m_n|^2$,
  \begin{align*}
    \frac{R_n^2}{2} = \frac{M_n^2 + 2M_n|m_n| + |m_n|^2}{2} \leq M_n^2 + |m_n|^2 \leq R_n^2,
  \end{align*}
  it also gives the bounds
  \begin{align}\tag{$\ast$}
    \Big|\frac{1-c}{2} J_2(c) \Big| \lesssim \Big|\lim_{n\to\infty}
    \mathbb E\Big[\Big(\frac{X_n}{\sqrt n}\Big)^2\Big] - 1 \Big|\lesssim
    \big|(1-c) J_2(c)\big|.
  \end{align}
  From Figure~\ref{fig:sim}, we see that the left hand side (LHS)
  performs better for $c>1$, while the right hand side (RHS) is better
  for $c<1$. The reason is that for $c>1$, the process is more likely
  to switch between its maximum and minimum, such that
  $M_n \approx |m_n| \approx R_n/2$, leading to
  $M_n^2 + |m_n|^2 \approx R_n^2/2$, while for $c<1$ switching becomes
  less likely and we rather have that $R_n \approx M_n$ or
  $R_n\approx |m_n|$, which gives $M_n^2 + |m_n|^2 \approx R_n^2$.
  \begin{figure*}[htb]
    \begin{center}
    \begin{tikzpicture}
      \node[anchor=south west,inner sep=0] (Bild) at (0,0) {\includegraphics[width=0.5\textwidth,trim=30 30 26 26,clip]{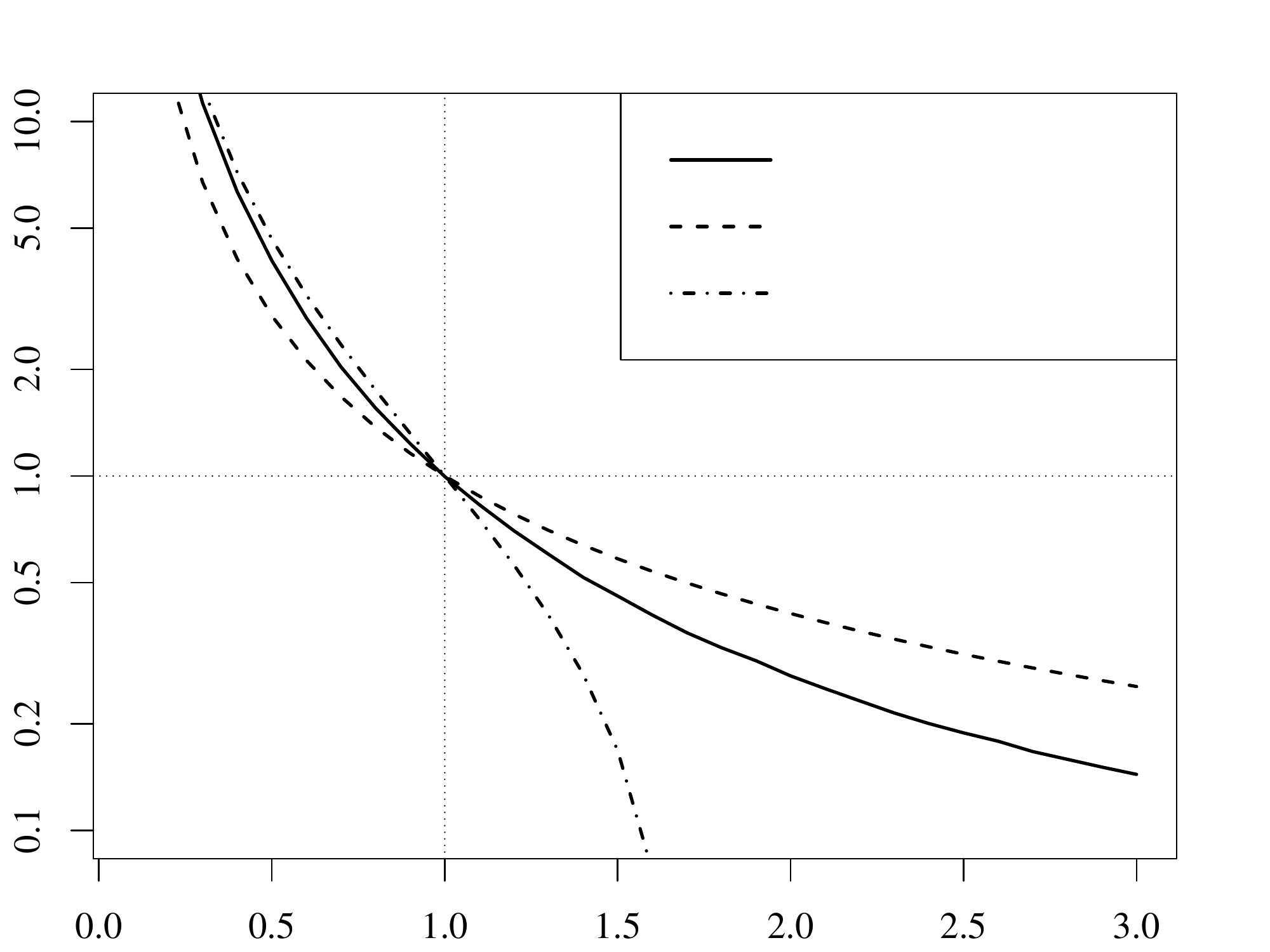}};
      \begin{scope}[x=(Bild.south east),y=(Bild.north west)]
        \draw (-.15, .5) node {\begin{sideways}$\mathbb V[X_n/\sqrt n]$\end{sideways}};
        \draw (.77, .89) node {simulation};
        \draw (.77, .8) node {LHS of ($\ast$)};
        \draw (.77, .71) node {RHS of ($\ast$)};
        \draw (0.45,-0.15) node {$c$ };
        \draw (0.03,-0.05) node {\small $0$ };
        \draw (0.18,-0.05) node {\small $0.5$ };
        \draw (0.33,-0.05) node {\small $1$ };
        \draw (0.48,-0.05) node {\small $1.5$ };
        \draw (0.63,-0.05) node {\small $2$ };
        \draw (0.78,-0.05) node {\small $2.5$ };
        \draw (0.93,-0.05) node {\small $3$ };
        \draw (-0.05,0.07) node[] {\small $0.1$};
        \draw (-0.05,0.2) node[] {\small $0.2$};
        \draw (-0.05,0.36) node[] {\small $0.5$};
        \draw (-0.05,0.5) node[] {\small $1.0$};
        \draw (-0.05,0.63) node[] {\small $2.0$};
        \draw (-0.05,0.79) node[] {\small $5.0$};
        \draw (-0.05,0.93) node[] {\small $10.0$};
      \end{scope}
    \end{tikzpicture} 
    \end{center}
    \caption{\label{fig:sim}Simulating the ORRW for varying
      $c$, we observe $10^5$ independent draws of $X_n/\sqrt n$ for
      $n=10^5$, and compute the observed variance. We compare this to
      the two bounds in ($\ast$).}
  \end{figure*}
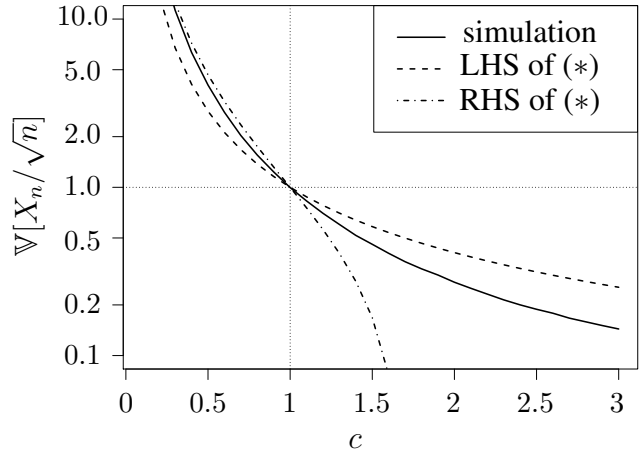
\end{remark}

\section{Proof of Theorem~\ref{T1}}
\label{S3}
\subsection{Some preliminairies}
\label{ss:31}
Before we come to the proof of Theorem~\ref{T1}, we need some general
results. First, in Theorem~\ref{T:tauber}, we recall a classical
Tauberian result by Hardy and Littlewood, which will help us to
interprete the generating function of $S_k$ from
\eqref{eq:SRWtimes}. Then, in Lemma~\ref{l:rw}, we recall the
generating function of hitting times for a simple symmetric random
walk.

\begin{theorem}[A Tauberian result]\label{T:tauber}
  Let $a_1, a_2,... \geq 0$ such that $\sum_{n=1}^\infty a_n x^n$
  converges for $|x|<1$. Suppose that for some $\alpha, A\geq 0$
  $$\sum_{n=1}^\infty a_n x^n \stackrel{x\uparrow 1} \sim \frac{A}{(1-x)^\alpha}.$$
  Then,
  $$ \sum_{k=1}^n a_n \stackrel{n\to\infty} \sim \frac{A}{\Gamma(\alpha + 1)}n^\alpha.$$
  Moreover, if $\alpha>1$, and $n\mapsto a_n$ is non-decreasing,
  \begin{align}
    \label{eq:T:tauber2}
    a_n \sim \frac{A\alpha}{\Gamma(\alpha+1)}n^{\alpha-1}.
  \end{align}
\end{theorem}

\begin{proof}
  The assertions are classical Tauberian results by Hardy and
  Littlewood; see e.g.\ Chapter~I.7.4 of \cite{Korevaar2013}. Another
  self-contained proof is given in Proposition~12.5.2 in
  \cite{Lawler2010}.
\end{proof}

\noindent
The following lemma is rather standard (see e.g. Chapter XIV.4 in
\citealp{Feller1}), but we provide a proof for completeness.

\begin{lemma}[Generating function of hitting times in simple symmetric
  random walk]\label{l:rw}
  Let $Z$ by a simple symmetric random walk with $Z_0=0$ and
  $a,b>0$. Define $T_{k} := \inf\{n: Z_n = k\}$, the first hitting
  time of $k \in \mathbb Z$, as well as $T:= T_{-a} \wedge T_b$ for
  $a,b\geq 0$. Then,
  \begin{align*}
    \mathbb E[s^{T}] & = \frac{e^{d_s a} + e^{d_s b}}{1 + e^{d_s (a+b)}}
  \end{align*}
  where
  \begin{align}
    \label{eq:cs}
    d_s & := \log\Big(\frac 1s\Big(1 + \sqrt{1-s^2}\Big)\Big).
  \end{align}
\end{lemma}

\begin{proof}
  Recall $\cosh(r) := \frac{e^r + e^{-r}}{2}$ and note that
  $\cosh(d_s) = 1/s$.
  For any $r\in\mathbb R$, using that
  $$\mathbb E \left[e^{r Z_n}\right] = \mathbb E \left[e^{-r
      Z_n}\right] = \prod_{k=1}^n \frac{e^r + e^{-r}}{2} =
  \cosh(r)^n,$$ the stochastic process
  $(e^{rZ_n} / \mathbb E[e^{rZ_n}])_{n=0,1,2,...}$ is a martingale.
  Therefore, using $r= \pm d_s$,
  \begin{align*}
    M_n & := e^{d_sa} \frac{e^{d_s Z_n}}{\mathbb{E}\left[e^{d_s Z_n}\right]} + e^{d_sb}
    \frac{e^{-d_s Z_n}}{\mathbb{E}\left[e^{-d_s Z_n}\right]} = \frac{e^{d_s(a+Z_n)} + e^{d_s(b-Z_n)}}{\cosh(d_s)^n}
    \\ & = s^n(e^{d_s(a+Z_n)} + e^{d_s(b-Z_n)})
  \end{align*}
  is a martingale as well. We apply the optional sampling theorem to
  the bounded martingale $M_{T\wedge n}$ to obtain
  \begin{align*}
    e^{d_s a}+e^{d_sb} &= \mathbb E \left[M_0\right] = \mathbb E\left[M_{T}\right]
                         = \mathbb E\left[s^T\left(1+e^{d_s(b+a)}\right)\right].
  \end{align*}
  From this, we read off the
  result.
\end{proof}

\subsection{Proof of Theorem~\ref{T1}}
\label{ss:32}
Our analysis of the moments of $R_n$ will be done via an analysis of
the generating function of the random variable $S_k$, which we defined
in~\eqref{eq:SRWtimes}. We start by computing the generating function
of $S_k$.

\begin{lemma}[Generating function of $\tau_i, T_i$ and
  $S_k$]\label{l:gentaui}
  Fix $s\in (0,1)$, recall $d_s$ from \eqref{eq:cs} and let
  \begin{align}
    \label{eq:gG} g_x(s):= \frac{e^{d_s} + e^{d_s(x-1)}}{1+e^{d_s x}}, \qquad
    G_x(s):= \frac{s}{1+c - cs g_x(s)}.
  \end{align}
  Then, for $\tau_i$ and $T_i$ as in \eqref{eq:SRWtimes}, the
  generating functions are given by
  \begin{align*}
    \mathbb E[s^{\tau_i}] = g_i(s), \qquad \mathbb E[s^{T_i}] = G_i(s), \quad i=2,3,...
  \end{align*}
  Moreover, the generating function of $S_k$ is given by
  \begin{align}\label{eq:genSk}
    \mathbb E[s^{S_k}] = s \prod_{i=1}^{k-1} G_i(s), \qquad k=1,2,...
  \end{align}
\end{lemma}

\begin{proof}
  The first claim follows directly from Lemma~\ref{l:rw}. For the
  generating function of $T_i$, note that the generating function of
  $Y_i \sim \text{geo}(1/(1+c))$ is
  $$ s\mapsto \frac{1}{c+1} \sum_{k=0}^\infty \Big(\frac{cs}{c+1}\Big)^k = \frac{1}{1+c}
  \cdot \frac{c+1}{c+1-cs} = \frac{1}{1 + c - cs},$$ hence
  \begin{align*}
    \mathbb E[s^{T_i}] & = s\mathbb E\Big[\mathbb E\Big[ \prod_{j=1}^{Y_i} s^{1+\tau_i^j} \Big|Y_i \Big]\Big]
                         = \mathbb E[(sg_i(s))^{Y_i}] = \frac{s}{1+c - csg_i(s)}.
  \end{align*}
  The form of the generating function of $S_k$ follows from
  \eqref{eq:SRWtimes}.
\end{proof}

\begin{lemma}[A generating function for moments of $R_n$]\label{l:Hls}
  Recall $G_x(s)$ from \eqref{eq:gG}. Then, for $s\in (0,1)$ and
  $\ell=0,1,2,...$
  $$ H_\ell(s) := \sum_{n=1}^\infty s^n \mathbb E[R_n \cdots (R_n + \ell)]
  = \frac{\ell+1}{1-s} \sum_{k=1}^\infty k \cdots (k+\ell-1)
  s\prod_{i=1}^{k-1} G_i(s).$$
\end{lemma}

\begin{proof}
  We will use, for $k, n=1,2,...$
  \begin{align*}
    \mathbb P(R_n=k) &= \mathbb P( R_n<k+1 ) - \mathbb P( R_n<k ) = \mathbb P( S_{k+1}>n ) - \mathbb P( S_k>n )
    \\ & = \mathbb P(S_k\le n ) - \mathbb P(S_{k+1}\le n).
  \end{align*}
  Then,
  \begin{align*}
    H_\ell(s) & 
           = \sum_{k=1}^\infty \sum_{n=1}^\infty s^n k \cdots (k+\ell) (\mathbb P(S_k\leq n) - \mathbb P(S_{k+1} \leq n))
    \\ & = \sum_{k=1}^\infty \sum_{n=1}^\infty s^n (k \cdots (k+\ell) - (k-1) \cdots (k+\ell-1)) \mathbb P(S_k\leq n)
    \\ & = (\ell+1) \sum_{k=1}^\infty \sum_{n=1}^\infty \sum_{i=1}^n s^n k\cdots (k+\ell-1) \mathbb P(S_k=i)
    \\ & = (\ell+1) \sum_{k=1}^\infty k\cdots (k+\ell-1)  \sum_{i=1}^\infty \mathbb P(S_k=i) \sum_{n=i}^\infty s^n
    \\ & = \frac{\ell+1}{1-s} \sum_{k=1}^\infty k\cdots (k+\ell-1) \sum_{i=1}^\infty \mathbb P(S_k=i) s^i
  \end{align*}
  and the result follows from \eqref{eq:genSk}.
\end{proof}

\begin{proposition}[Asymptotics of $H_\ell$]\label{P:Hell}
  Let $\ell=0,1,2,...$ and $H_\ell$ as in Lemma~\ref{l:Hls}. Then,
  \begin{align*}
    H_\ell(s) \stackrel{s\uparrow 1}\sim K_\ell (1-s)^{-(3+\ell)/2},
  \end{align*}
  where
  $$ K_\ell = \frac{\ell+1}{2^{(\ell+1)/2}} 2^{2c} 
  \int_0^\infty x^{\ell} \Big(\frac{e^x}{(e^x+1)^2} \Big)^c dx.$$
\end{proposition}

\begin{proof}
  We define
  $$ a: x\mapsto \frac{e^x-1}{e^x+1}$$
  and write for $g_.(s)$ as in \eqref{eq:gG}, using the definition of
  $d_s$,
  \begin{align*}
    g_{x}(s) & =  \frac{(1 + \sqrt{1-s^2})/s + e^{d_s x} s/(1+\sqrt{1-s^2})}{1+e^{d_s x}}
    \\ & = \frac 1s \Big(1 - \frac{e^{d_s x} - \sqrt{1-s^2} - e^{d_s x} s^2/(1 + \sqrt{1-s^2})}{1+e^{d_s x}}\Big)
    \\ & = \frac 1s \Big(1 - \frac{e^{d_s x} \frac{1 + \sqrt{1-s^2} - s^2}{1+\sqrt{1-s^2}} - \sqrt{1-s^2}}{1+e^{d_s x}}\Big)
    \\ & = \frac 1s \Big(1 - \sqrt{1-s^2}\frac{e^{d_s x} - 1}{e^{d_s x}+1}\Big)
         = \frac 1s\Big(1 - a(d_s x)\sqrt{1-s^2}\Big).
  \end{align*}
  as well as
  \begin{align*}
    G_{x}(s) & = \frac{s}{1+ca(d_s x) \sqrt{1-s^2} }.
  \end{align*}
  In the sequel, we will use $t:=t(s) := 1-s$ throughout, set
  $$ f_t := d_{1-t} = \text{arcosh}\Big(\frac{1}{1-t}\Big) = \log\Big(\frac{1}{1-t}\Big(1 + \sqrt{t(2-t)}\Big)\Big)$$
  and note that
  \begin{align*}
    f_t & = - \log(1-t) + \log\Big(1 + \sqrt{t(2-t)}\Big) = t  + \sqrt{t(2-t)} - \tfrac t2(2-t) + O(t^{3/2})
    \\ & = \sqrt{2t} + O(t^{3/2})
  \end{align*}
  as $t\to 0$. Therefore,
  \begin{align*}
    g_x(1-t) & = \frac 1{1-t}\Big(1 - a(f_tx)\sqrt{t(2-t)}\Big)
    \\ & = (1+O(t))\Big(1 - a(\sqrt{2t}x)\sqrt{2t}(1 + O(t))\Big)
    \\ & = 1 - a(\sqrt{2t}x) \sqrt{2t} + O(t)
         \intertext{and}
         G_x(1-t) & = \frac{1-t}{1 + c a(f_tx)\sqrt{t(2-t)}}
                    = (1-t)(1-c a(\sqrt{2t}x) \sqrt{2t}(1 + O(t)) + O(t))
    \\ & = 1-ca(\sqrt{2t}x) \sqrt{2t} + O(t).
  \end{align*}
  The latter can now be used for (note that the empty product arising
  for $\ell=0$ is defined to be~1)
  \begin{align*}
    H_\ell(1-t) & t^{(3 + \ell)/2} = \frac{\ell+1}{2^{(\ell+1)/2}} \sum_{k=1}^\infty \sqrt{2t}
                  \Big(\prod_{j=0}^{\ell-1} \sqrt{2t}(k+j)\Big) (1-t)\exp\Big( \sum_{i=1}^{k-1}
                  \log G_i(1-t)\Big)
    \\ & = \frac{\ell+1}{2^{(\ell + 1)/2}} \sum_{k=1}^\infty \sqrt{2t} \Big(\prod_{j=0}^{\ell-1} \sqrt{2t}(k+j)\Big) \exp\Big(
         - \sum_{i=1}^{k-1} c a(\sqrt{2t}i) \sqrt{2t} + O(t)\Big)
    \\ & = \frac{\ell+1}{2^{(\ell + 1)/2}} \int_0^\infty x^\ell \exp\Big(
         - \int_0^x c a(y) dy\Big) dx \cdot (1 + o(1)),
  \end{align*}
  where we have used two approximations of integrals with Riemann-sums
  (with $\sqrt{2t} \approx dx, dy$). The result follows from
  \begin{align*}
    & \int_0^x a(y) dy \stackrel{z=e^y}= \int_1^{e^x} \frac{z-1}{z(z + 1)} dz =
      \int_1^{e^x} \frac{2}{z+1} - \frac{1}{z} dz = 2\log(e^x + 1) - 2\log 2 - x
      \intertext{and}
    & \exp\big(-c(2\log(e^x+1) - 2\log 2 - x)\big) = 2^{2c}\Big(\frac{e^x}{(e^x+1)^2} \Big)^c.
  \end{align*}
\end{proof}

\begin{proof}[Proof of Theorem~\ref{T1}]
  We now combine the results of Proposition~\ref{P:Hell} with the
  Tauberian result from Theorem~\ref{T:tauber}. We obtained in
  Proposition~\ref{P:Hell} for
  $a_n = \mathbb E[R_n \cdots (R_n + \ell)]$ that
  \begin{align*}
    H_\ell(s) = \sum_{n=1}^\infty a_n s^n \stackrel{s\uparrow 1} \sim \frac{K_\ell}{(1-s)^{(3+\ell)/2}}
  \end{align*}
  and we can apply Theorem~\ref{T:tauber} with $A = K_\ell$ and
  $\alpha = (3 + \ell)/2$. In particular, \eqref{eq:T:tauber2} gives
  that
  \begin{align*}
    \mathbb E[R_n \cdots (R_n + \ell)] \sim \frac{K_\ell (3 + \ell)}{2\Gamma((5+\ell)/2)} n^{(1+\ell)/2}.
  \end{align*}
  Since $\Gamma(x+1) = x\Gamma(x)$, this implies
  \begin{align*}
    \mathbb E[R_n^{\ell+1}] & \sim \frac{1}{2^{(\ell-1)/2}
                              \Gamma((\ell+1)/2)} \cdot J_{\ell+1}(c) \cdot n^{(\ell+1)/2}, \qquad \ell=0,1,2,...
                              \intertext{with}
                              J_\ell(c) & := 2^{2c} \int_0^\infty \Big( \frac{e^x}{(e^x+1)^2}\Big)^c dx
  \end{align*}
  and we are done.
\end{proof}

\subsubsection*{Acknowledgements}
We thank Tanja Schilling for introducing us to the model of the hungry
random walk.

\end{document}